\documentclass[9pt, journal]{IEEEtran}
\usepackage{mathtools}
\usepackage{bm}
\usepackage{color}
\usepackage{amsmath, amsthm,amssymb,latexsym}
\newtheorem{theorem}{Theorem}
\newtheorem{lemma}[theorem]{Lemma}
\newtheorem{assumption}[theorem]{Assumption}

\newtheorem{remark}[theorem]{Remark}
\newtheorem{corollary}[theorem]{Corollary}

\usepackage{pgfplots}
\pgfplotsset{compat=newest}
\usetikzlibrary{plotmarks}
\usepgfplotslibrary{patchplots}
\usepackage{grffile}
\usepackage{amsmath}
\allowdisplaybreaks
\usepackage{tikz}
\usetikzlibrary{graphs} \ifnum\pdfshellescape=1
 \fi

\usetikzlibrary{positioning}
\usetikzlibrary{graphs}

\newcommand{\E}[1]{\mathbb{E}\left\{ #1 \right\} }

\begin{document}

\title{Distributed Consensus over Markovian Packet Loss Channels}

 \author{Liang Xu, Yilin Mo, Lihua Xie 
   \thanks{Liang Xu, Yilin Mo and Lihua Xie are with the School of
     Electrical and Electronic Engineering, Nanyang Technological University,
     Singapore 639798, Singapore. Email:{\tt\small lxu006@e.ntu.edu.sg, \{ylmo,
       elhxie\}@ntu.edu.sg} }}

\date{\today}

\maketitle

\begin{abstract}
This paper studies the consensusability problem of multi-agent systems (MASs), where agents communicate with each other through Markovian packet loss channels.  We try to determine conditions under which there exists a linear distributed consensus controller such that the MAS can achieve mean square consensus. We first provide a necessary and sufficient consensus condition for MASs with single input and i.i.d.\ channel losses, which complements existing results. Then we proceed to study the case with identical Markovian packet losses. A necessary and sufficient consensus condition is firstly derived based on the stability of Markov jump linear systems. Then a numerically verifiable consensus criterion in terms of the feasibility of linear matrix inequalities (LMIs) is proposed. Furthermore, analytic sufficient conditions and necessary conditions for mean square consensusability are provided for general MASs. The case with nonidentical packet loss is studied subsequently. The necessary and sufficient consensus condition and a sufficient consensus condition in terms of LMIs are proposed. In the end, numerical simulations are conducted to verify the derived results.
\end{abstract}

\section{Introduction}

The rapid development of technology has enabled wide applications of multi-agent systems (MASs). The consensus problem, which requires all agents to agree on certain quantity of common interests, builds the foundation of other cooperative tasks. One question arises before control synthesis: whether there exist distributed controllers such that the MAS can achieve consensus. This problem is referred to as consensusability of MASs. Previously, the consensusability problem with perfect communication channels has been well studied under an undirected/directed communication topology~\cite{MaCuiqin2010TAC, YouKeyou2011TACconsensusability, GuGuoxiang2012TAC, LiZhongkui2010TCS, TrentelmanHL2013TAC}. In~\cite{MaCuiqin2010TAC}, it is shown that to ensure the consensus of a continuous-time linear MAS, the linear agent dynamics should be stabilizable and detectable, and the undirected communication topology should be connected. Furthermore, references~\cite{YouKeyou2011TACconsensusability,GuGuoxiang2012TAC} show that for a discrete-time linear MAS, the product of the unstable eigenvalues of the agent system matrix should additionally be upper bounded by a function of the eigen-ratio of the undirected graph. Extensions to directed graphs and robust consensus can be found in~\cite{LiZhongkui2010TCS,TrentelmanHL2013TAC}.

Most of the consensusability results discussed above are derived under perfect communications assumptions. However, this is not the case in practical applications, where communication channels naturally suffer from limited data rate constraints, signal-to-noise ratio constraints, time-delay and so on. Therefore, the consensusability problem of MASs under communication channel constraints has been widely studied in~\cite{LiuShuai2011SIAM, QiuZhirong2017TAC, LiZhongkui2017TAC, XuLiang2016Automatica,  QiTian2016TAC, WangZhenhua2017ScienceChina} under different channel models. In this paper, we are interested in the lossy channel~\cite{SinopoliBruno2004TAC,MoYilin2012TAC,YouKeyou2011TACmarkovian}, which models the packet drop phenomenon in wireless communications due to the communication noise, interference or congestion. Previously, the case with independent and identically distributed (i.i.d.) channel fading has been studied in~\cite{XuLiang2016Automatica}. However, the i.i.d.\ assumption fails to capture the correlation of channel conditions over time. Since Markov models are simple and effective in capturing temporal correlations of channel conditions~\cite{GoldsmithA2005Book,Huang2007Automatica}, we are interested in the consensusability problem of MASs over Markovian packet loss channels. Due to the existence of correlations of packet losses over time, the methods used to deal with the i.i.d.\ channel state  in~\cite{XuLiang2016Automatica} cannot be applied directly to the Markovian channel loss case. 

This paper studies the consensusability problem of MASs over packet loss channels. The contributions are listed as follows. A necessary and sufficient consensus condition for MASs with single input and identical i.i.d.\ channel losses is firstly derived, which complements existing results and explicitly demonstrates how the network topology, the agent dynamics and the packet loss interplay with each other in the consensus problem. For the consensus with identical Markovian packet loss: 1, a necessary and sufficient consensusability condition is provided; 2, a numerically testable criterion and analytical sufficient and necessary consensusability conditions are derived; 3, a critical consensusability condition is obtained for the special case of scalar agent dynamics. For the consensus with nonidentical Markovian packet loss, sufficient and necessary consensus conditions are also derived by introducing the edge Laplacian.

Some preliminaries results on distributed consensus over identical Markovian packet loss channels are contained in~\cite{XuLiangNecSys2018}. This paper contains new results for the case of MASs with single input and identical i.i.d.\ packet losses, and for distributed consensus with nonidentical Markovian packet loss. This paper is organized as follows: The problem formulation is stated in Section~\ref{sec.ProblemFormulation}. The consensusability results for MASs with single input and identical i.i.d.\ channel losses are presented in Section~\ref{sec.SingleInputIIDLoss}. The consensusability results for the cases with identical Markovian and nonidentical Markovian packet losses are discussed in Section~\ref{sec.IdenticalMarkovLoss} and Section~\ref{sec.NonidenticalMarkovLoss}, respectively. Numerical simulations are provided in Section~\ref{sec.Simulations}. This paper ends with some concluding remarks in Section~\ref{sec.Conclusion}. 

\textit{Notation}: All matrices and vectors are assumed to be of appropriate dimensions that are clear from the context. $\mathbb{R}, \mathbb{R}^{n}, \mathbb{R}^{m\times n}$ represent the sets of real scalars, $n$-dimensional real column vectors and $m\times n$-dimensional real matrices, respectively. $\bm{1}$ denotes a column vector of ones. $I$ represents the identity matrix. $A'$, $A^{-1}$, $\rho(A)$ and $\mathrm{det}(A)$ are the transpose, the inverse, the spectral radius and the determinant of matrix $A$, respectively. $\otimes$ represents the Kronecker product. For a symmetric matrix $A$, $A \ge 0$ ($A>0$) means that matrix $A$ is positive semi-definite (definite). For a symmetric matrix $A$, $\lambda_{\mathrm{min}}(A)$ denotes the smallest eigenvalue of $A$. $\mathrm{diag}(A, B)$ denotes a diagonal matrix with diagonal entries $A$ and $B$.  $\mathbb{E}\{\cdot\}$ denotes the expectation operator. The symmetric matrix $ \begin{bsmallmatrix} A & C'\\ C &B \end{bsmallmatrix} $ is abbreviated as  $ \begin{bsmallmatrix} A & * \\ C &B \end{bsmallmatrix} $.

\section{Problem Formulation \label{sec.ProblemFormulation}}

Let $\mathcal{V}=\{1, 2, \ldots, N\}$ be the set of $N$ agents with $i\in \mathcal{V}$ representing the $i$-th agent. A graph $\mathcal{G}=(\mathcal{V} ,\mathcal{E})$ is used to describe the interaction among agents, where $\mathcal{E} \subseteq \mathcal{V}\times \mathcal{V} $ denotes the edge set with paired agents. We assume $\mathcal{G}$ is undirected throughout this paper. An edge $(j,i) \in \mathcal{E}$ means that the $i$-th agent and the $j$-th agent can communicate each other. The neighborhood set $\mathcal{N}_i$ of agent $i$ is defined as $\mathcal{N}_i=\{j\;|(j,i)\in\mathcal{E}\}$. The graph Laplacian matrix $\mathcal{L}={[\mathcal{L}_{ij}]}_{N\times N}$ is defined as $\mathcal{L}_{ii}=\sum_{j\in \mathcal{N}_i}a_{ij}$, $\mathcal{L}_{ij}=-a_{ij}$ for $i\neq j$, where $a_{ii}=0$, $a_{ij}=1$ if $(j,i)\in \mathcal{E}$ and $a_{ij}=0$, otherwise. A path on $\mathcal{G}$ from agent $i_1$ to agent $i_l$ is a sequence of ordered edges in the form of $(i_{k},i_{k+1})\in \mathcal{E}$, $k=1,2,\ldots, l-1$. A graph is connected if there is a path between every pair of distinct nodes.

In this paper, we assume that each agent has the homogeneous dynamics
\begin{align}
  \label{eq:AgentDynamics} x_{i}(t+1)=Ax_{i}(t)+Bu_i(t), \quad i=1, \ldots, N,
\end{align}
where $x_i\in \mathbb{R}^n$ is the system state; $u_i \in \mathbb{R}^m$ is the control input; $(A, B)$ is controllable and $B$ has full-column rank. The interaction among agents is characterized by an undirected connected graph $\mathcal{G}=\{\mathcal{V, \mathcal{E}}\}$. The consensus protocol is given by
\begin{align}
  \label{eq:ConsensusProtocol}
  u_i(t)=\sum_{j\in \mathcal{N}_i} \gamma_{ij}(t) K(x_i(t)-x_j(t)),  
\end{align}
where $\gamma_{ij}(t)\in \{0, 1\}$ models the lossy effect of the communication channel from agent $j$ to agent $i$, which satisfies that $\gamma_{ij}(t)=1$ when the transmission is successful at time $t$, and $0$ otherwise. 

Throughout the paper, we say that the MAS~\eqref{eq:AgentDynamics} is mean square consensusable by the protocol~\eqref{eq:ConsensusProtocol} if there exists $K$ such that the MAS~\eqref{eq:AgentDynamics} can achieve mean square consensus under the protocol \eqref{eq:ConsensusProtocol}, i.e., $\lim_{t \rightarrow \infty} \E{\|x_i(t)-x_j(t)\|^2}=0$ for all $i, j\in \mathcal{V}$. The following assumption is made as in~\cite{YouKeyou2011TACconsensusability}.
\begin{assumption}\label{assumption:matrixAisUnstable}
    All the eigenvalues of $A$ are either on or outside the unit disk.
\end{assumption}

\section{Single Input Agent Dynamics with I.I.D. Packet Loss \label{sec.SingleInputIIDLoss}}

In this section, we consider the special case that the agent is with single input, i.e., $u_i\in \mathbb{R}$ and the packet loss processes for all channels are identical and i.i.d., which has been studied in~\cite{XuLiang2016Automatica}. We provide a necessary and sufficient condition to guarantee the mean square consensus, which complements existing results in~\cite{XuLiang2016Automatica}, where only the sufficiency is proved. Specifically, we make the following assumption about the packet loss process. 

\begin{assumption}\label{assumption:iidPacketLoss}
$\gamma_{ij}(t)=\gamma(t)$ for all $(i, j)\in \mathcal{E}$ and  $t\ge 0$. Moreover, the sequence $\{ \gamma(t) \}_{t\ge 0} $ is i.i.d.\ and $\gamma(t)$ has a Bernoulli distribution with lossy rate $p$.
\end{assumption}

Define the consensus error as $\delta(t)= (I-\frac{1}{N}\bm{1}\bm{1}')x(t)$,
where $x(t)=[x_1(t)', \ldots, x_N(t)']'$. Following similar derivations as in~\cite{XuLiang2016Automatica}, the consensus error dynamics is given by
\begin{align}
  \label{eq:ConsensusErrorDynamics}
 \delta(t+1)=(I\otimes A + \gamma(t)\mathcal{L}\otimes BK)\delta(t), 
\end{align}
where $\mathcal{L}$ is the graph Laplacian of $\mathcal{G}$. If there exists $K$ such that system~\eqref{eq:ConsensusErrorDynamics} is mean square stable, i.e., $ \lim_{t\rightarrow \infty} \E{ \delta(t)\delta(t)' }=0$, the MAS can achieve mean square consensus. It has been proved in~\cite{XuLiang2016Automatica} that the mean square stability of~\eqref{eq:ConsensusErrorDynamics} is equivalent to the simultaneous mean square stability of
\begin{align}\label{eq.SimultaneousMJLS}
 \delta_i(t+1)=(A+\lambda_i \gamma(t)BK)\delta_i(t), \quad i=2, \ldots, N,
\end{align}
where $\lambda_i$ is the nonzero positive eigenvalues of $\mathcal{L}$ with $\lambda_2\le \dots \le \lambda_N$. The following lemmas are needed in the proof of the main result and stated first.
\begin{lemma}
\label{lem.existenceLemma}
Let $P > 0$. Suppose there exists a vector $v$, such that $v'Pv > \phi^2$ and $\phi > 0$, then there exists a vector $x$, such that the following inequalities hold
\begin{align*}
 x'P^{-1}x < 1,\, x'v > \phi.
\end{align*}
\end{lemma}
\begin{proof}
Let us choose $x =\alpha Pv$, then
\begin{align*}
 x'P^{-1}x = \alpha^2 v'Pv,\,x'v = \alpha v'Pv. 
\end{align*}
Since $v'Pv > \phi^2$, we can choose $\alpha$, such that
\begin{align*}
 \frac{\phi}{v'Pv}< \alpha < \frac{1}{\sqrt{v'Pv}}.
\end{align*}
The proof is completed.
\end{proof}

\begin{lemma}
    \label{lem.ARESolutionProperty}
Suppose that $\mathrm{det}(A)\neq 0$, then any $P>0$ that satisfies
\begin{align}
\label{eq:ARE}
   P - A'PA + A'PB(B'PB)^{-1}B'PA > 0,
\end{align}
must also satisfy
\begin{align}
  \label{eq.ARESolutoinSatisfy}
 \frac{B'(A')^{-1}PA^{-1}B}{B'PB} \le \frac{1}{\det(A)^2}. 
\end{align}
\end{lemma}
\begin{proof}
    We prove this lemma by contradiction. Let
    \begin{align*}
      g(P) = P - A'PA + A'PB(B'PB)^{-1}B'PA > 0.
    \end{align*}
    Suppose that there exists a $P>0$ to~\eqref{eq:ARE}, such that
    \begin{align*}
      \frac{B'(A')^{-1}PA^{-1}B}{B'PB} > \frac{1}{(\det{A})^2},
    \end{align*}
    then we have
    \begin{align*}
      B'(A')^{-1}\frac{g(P)}{B'PB}A^{-1}B= \frac{B'(A')^{-1}PA^{-1}B}{B'PB} > \frac{1}{(\det{A})^2}.
    \end{align*}
Therefore, by Lemma~\ref{lem.existenceLemma}, there exists a $K$, such that
\begin{align}
 (K-K^*) \left(\frac{g(P)}{B'PB}\right)^{-1} (K-K^*)' < 1,\label{eq.Contradict1}
\end{align}
and
\begin{align}
  (K-K^*)A^{-1}B > \left|\frac{1}{\det A}\right|, \label{eq.Contradict2}
\end{align}
where $K^* = -B'PA/(B'PB)$.

By Schur complement lemma and the fact that $B'PB > 0$,~\eqref{eq.Contradict1} is equivalent to
\begin{multline*}
P >  A'PA - A'PB(B'PB)^{-1}B'PA \\+ (K-K^*)'B'PB(K-K^*),
\end{multline*}
or $P>(A+BK)'P(A+BK)$. Therefore, $K$ is stabilizing, i.e., $A+BK$ is strictly stable. By matrix determinant lemma, if $A+BK$ is stable, then
$ |\det(A+BK)| = |\det A||1+KA^{-1}B| < 1$.
Therefore, 
\begin{align}
  KA^{-1}B < -1+\left|\frac{1}{\det A}\right|.
  \label{eq:maxKAinvB}
\end{align}

On the other hand, by the definition of $K^*$, we have
\begin{align*}
 K^*A^{-1}B  = -\frac{B'PAA^{-1}B}{B'PB} = -1.
\end{align*}
Then, we have from~\eqref{eq.Contradict2} that
\begin{align*}
  KA^{-1}B > -1+\left|\frac{1}{\det A}\right|, 
\end{align*}
which contradicts with~\eqref{eq:maxKAinvB}. Therefore, any $P>0$ that satisfies~\eqref{eq:ARE}, must also satisfy \eqref{eq.ARESolutoinSatisfy}. 
The proof is completed.
\end{proof}

The main result is stated as follows.
\begin{theorem}
    \label{thm.SingleInputIIDIffCondition}
 Under Assumption~\ref{assumption:matrixAisUnstable} and~\ref{assumption:iidPacketLoss}, when $m=1$, the MAS~\eqref{eq:AgentDynamics} is mean square consensusable by the protocol~\eqref{eq:ConsensusProtocol} if and only if
   \begin{align}
     \label{eq.SingleInputIIDIffCondition}
     (1-p)\left[ 1-\left( \frac{\lambda_N-\lambda_2}{\lambda_N+\lambda_2} \right)^2 \right]>1-\frac{1}{(\det A)^2}. 
   \end{align}
\end{theorem}

\begin{proof}
The sufficiency follows from Theorem 1 in~\cite{XuLiang2016Automatica}. Only the necessity is proved here, which follows from the simultaneous mean square stability of~\eqref{eq.SimultaneousMJLS}. Let $\mu$ and $\sigma^2$ be the mean and variance of $\gamma(t)$, respectively, then we have $\mu=1-p$ and $\sigma^2=p(1-p)$.
   
In view of Lemma 1 in~\cite{XiaoNan2012TAC},~\eqref{eq.SimultaneousMJLS} is mean square stable for i.i.d.\ $\{\gamma(t)\}_{t\ge 0}$ if and only if there exist $P_i>0, i=2, \ldots, N$ and $K$, such that
\begin{align*}
 P_i>(A+\lambda_i\mu BK)'P_i(A+\lambda_i\mu BK) + \lambda_i^2 \sigma^2 K'B'P_iBK 
\end{align*}
for all $i=2, \ldots, N$. With some manipulations, we can show that
\begin{multline*}
    P_i-A'P_iA+\frac{\mu^2}{\mu^2+\sigma^2} A'P_iB(B'P_iB)^{-1}B'P_iA \\
    > \lambda_i^2
  (\mu^2+\sigma^2)
  \left(K+\frac{\mu}{\lambda_i(\mu^2+\sigma^2)}(B'P_iB)^{-1}B'P_iA\right)'\\
 \times B'P_iB\left(K+ \frac{\mu}{\lambda_i(\mu^2+\sigma^2)}(B'P_iB)^{-1}B'P_iA\right).
\end{multline*}
Left and right multiply the above inequality with $B'(A')^{-1}$ and $A^{-1}B$,
we can obtain that
\begin{multline}
  \label{eq:Constraints}
 \left(\lambda_i\sqrt{\mu^2+\sigma^2} KA^{-1}B+\frac{\mu}{\sqrt{\mu^2+\sigma^2}}\right)^2 \\< \frac{B'(A')^{-1}P_iA^{-1}B}{B'P_iB}+\frac{\mu^2}{\mu^2+\sigma^2}-1. 
\end{multline}
Since $P_i$ satisfies~\eqref{eq:ARE}, we have from Lemma~\ref{lem.ARESolutionProperty} that  
\begin{align*}
 \frac{B'(A')^{-1}P_iA^{-1}B}{B'P_iB}\le \frac{1}{\det(A)^2}. 
\end{align*}
Therefore we have from~\eqref{eq:Constraints} that
\begin{multline*}
{\left(
\lambda_i\sqrt{\mu^2+\sigma^2} KA^{-1}B+\frac{\mu}{\sqrt{\mu^2+\sigma^2}}
\right)}^2 \\< \frac{1}{\det(A)^2}   + \frac{\mu^2}{\mu^2+\sigma^2} - 1,
\end{multline*}
which further indicates
\begin{equation}
\label{eq:f0Exist}
\underline{\beta}_i<    \left|KA^{-1}B \right| < \overline{\beta}_i
\end{equation}
with
\begin{equation*}
\begin{aligned}
\underline{\beta}_i&=\frac{ -\sqrt{ \frac{1}{a_0^2} +
\frac{\mu^2}{\mu^2+\sigma^2} -1}
+\frac{\mu}{\sqrt{\mu^2+\sigma^2}} } {\lambda_i \sqrt{\mu^2+\sigma^2}},  \\
\overline{\beta}_i&= \frac{ \sqrt{ \frac{1}{a_0^2} +
\frac{\mu^2}{\mu^2+\sigma^2} -1}
+\frac{\mu}{\sqrt{\mu^2+\sigma^2}} } {\lambda_i
\sqrt{\mu^2+\sigma^2}},
\end{aligned}
\end{equation*}
where $a_0=\det(A)$.

Since there exists a common $\left| KA^{-1}B \right|$, such that~\eqref{eq:f0Exist} holds for all $i = 2, \ldots, N$. $ \cap_i \bigl(\underline{\beta}_i ,\overline{\beta}_i \bigr)$ must be non-empty, which implies $\underline{\beta}_2 <\overline{\beta}_N $. Further calculation shows that
\begin{align}
  \frac{\mu^2}{\mu^2+\sigma^2}\times
\left[
    1-{\left(
          \frac{\lambda_N-\lambda_2}{\lambda_N+\lambda_2}
      \right)}^2
\right] > 1-\frac{1}{a_0^2}. \label{eq.ConsensusConditionForGeneralFading}
\end{align}
Substituting the definitions of $\mu$, $\sigma^2$ and $a_0$, we can obtain~\eqref{eq.SingleInputIIDIffCondition}.
The proof is completed. \end{proof}

Furthermore, in view of the above derivations, we have the following consensusability condition for the general fading case studied in~\cite{XuLiang2016Automatica}.
\begin{corollary}
Under Assumption~\ref{assumption:matrixAisUnstable}, if $\gamma_{ij}(t)=\gamma(t)$  for all $(i,j)\in\mathcal{E}$ and  $\{ \gamma(t) \}_{t\ge 0} $ is i.i.d.\ with mean $\mu$ and variance $\sigma^2$, when $m=1$, the MAS~\eqref{eq:AgentDynamics} is mean square consensusable by the protocol~\eqref{eq:ConsensusProtocol} if and only if~\eqref{eq.ConsensusConditionForGeneralFading} holds. \end{corollary}


\section{Identical Markovian Packet Loss \label{sec.IdenticalMarkovLoss}}

In the section, we consider a more general case that $u_i$ is a $\mathbb{R}^m$ vector with $m \ge 1$ and $\gamma(t)$ is a Markov process and make the following assumption.

\begin{assumption}
    \label{assump.identicalPacketLoss}
    $\gamma_{ij}(t)=\gamma(t)$ for all $(i, j)\in \mathcal{E}$ and  $t\ge 0$. Moreover, $\{\gamma(t)\}_{t\ge 0}$ is a time-homogeneous Markov process with two states $\{0, 1\}$ and the transition probability matrix $Q$ is 
\begin{equation}
    \label{eq:twostateQ}
    Q =
    \left[ 
    \begin{matrix}
        1-q & q \\
        p   & 1-p
        \end{matrix} \right],
\end{equation}
where $0<p<1$ represents the failure rate and $0<q<1$ denotes the recovery rate.
\end{assumption}

\begin{remark}
Markov models are widely used to capturing temporal correlations of channel conditions~\cite{GoldsmithA2005Book, Huang2007Automatica}. However, due to the correlations of packet losses over time, the methods used to deal with the i.i.d.\ channel fading in~\cite{XuLiang2016Automatica} cannot be applied to the Markovian packet loss case.
\end{remark}

Since $\{\gamma(t)\}_{t\ge 0}$ is a Markov process, the consensusability is equivalent to the simultaneous mean square stabilizability of the $N-1$ Markov jump linear systems~\eqref{eq.SimultaneousMJLS}. In view of Theorem 3.9 in~\cite{CostaO2005Book} describing the stability of Markov jump linear systems, we can obtain the following consensusability condition.

\begin{theorem} \label{thm.IffCondition} Under Assumption~\ref{assumption:matrixAisUnstable} and~\ref{assump.identicalPacketLoss}, the MAS~\eqref{eq:AgentDynamics} is mean square consensusable by the protocol~\eqref{eq:ConsensusProtocol} if and only if either of the following conditions holds

    \begin{enumerate}
      \item There exist $K$, $P_{i,1}>0$, $P_{i,2}>0$ with $i=2, \ldots, N$, such that
\begin{gather*}
    \begin{multlined}[t][0.9\columnwidth]
        P_{i,1}- (1-q) A' P_{i,1}A \\
        - q (A+ \lambda_i BK)' P_{i,2} (A+\lambda_iBK)>0,     
    \end{multlined}
   \\
   \begin{multlined}[t][0.9\columnwidth]
       P_{i,2}- pA' P_{i,1}A \\
       - (1-p) (A+ \lambda_i BK)' P_{i,2}(A+\lambda_iBK)>0.
   \end{multlined}
\end{gather*}
\item There exists $K$ such that
\begin{align*}
\rho \left(\mathcal{H}_i\right)<1,
\end{align*}
for all $i=2, \ldots, N$ with
    \end{enumerate}
\begin{multline*}
  \mathcal{H}_i=\begin{bmatrix}
     (1-q) A\otimes A & p (A+\lambda_i BK)\otimes(A+\lambda_i BK) \\
     q A\otimes A & (1-p)(A+\lambda_i BK)\otimes (A+\lambda_i BK) 
 \end{bmatrix}.
\end{multline*}

\end{theorem}

With similar transformations as in the proof of Theorem~\ref{thm.NumericalSufficieny}, the consensus criterion in Theorem~\ref{thm.IffCondition}.1) can be shown to be equivalent to a feasibility  problem with bilinear matrix inequality (BMI) constraints. It is well known that checking the solvability of a BMI, is generally NP-hard~\cite{TokerOnur1995ACC}. Therefore, in the sequel, we propose a sufficient consensus condition in terms of the feasibility of linear matrix inequalities (LMIs) by a fixed $P_{i,1}$ and $P_{i,2}$.

\begin{theorem}
    \label{thm.NumericalSufficieny}
 Under Assumption~\ref{assumption:matrixAisUnstable} and~\ref{assump.identicalPacketLoss},
if there exist $Q_1>0$, $Q_2>0$, $Z_1$, $Z_2$ such that the following LMIs hold, 
\begin{gather}
  \begin{bmatrix}
      Q_1                  & *   & *   & * \\
      \sqrt{qc}(AQ_1+BZ_1) & Q_2 & *   & * \\
      \sqrt{q(1-c)}A Q_1   & 0   & Q_2 & * \\
      \sqrt{1-q}A Q_1      & 0   & 0   & Q_1 
  \end{bmatrix}>0, \label{eq.LMITest1}\\
  \begin{bmatrix}
      Q_2                      & *   & *   & * \\
      \sqrt{(1-p)c}(AQ_2+BZ_2) & Q_2 & *   & * \\
      \sqrt{(1-p)(1-c)}AQ_2    & 0   & Q_2 & * \\
      \sqrt{p}AQ_2             & 0   & 0   & Q_1 
  \end{bmatrix}>0,\label{eq.LMITest2}
\end{gather}
where $c=1- \left( \frac{\lambda_N-\lambda_2}{\lambda_N+\lambda_2} \right)^2>0$, then the MAS~\eqref{eq:AgentDynamics} is mean square consensusable by the protocol~\eqref{eq:ConsensusProtocol} and an admissible control gain is given by
\begin{align*}
 K=-\frac{2}{\lambda_2+\lambda_N}(B'Q_2^{-1}B)^{-1}B'Q_2^{-1}A. 
\end{align*}
\end{theorem}

\begin{proof}
If there exist $Q_1>0$, $Q_2>0$, $Z_1$, $Z_2$ such that \eqref{eq.LMITest1} and~\eqref{eq.LMITest2} hold, then there exist $P_1=Q_1^{-1}>0$, $P_2=Q_2^{-1}>0$, $K_1=Z_1P_1$ and $K_2=Z_2P_2$ such that
\begin{gather}
  \begin{bmatrix}
      P_{1}^{-1}              & *        & *        & * \\
   \sqrt{qc}(A+BK_1) P_1^{-1} & P_2^{-1} & *        & * \\
     \sqrt{q(1-c)}A P_1^{-1}  & 0        & P_2^{-1} & * \\
        \sqrt{1-q} AP_1^{-1}  & 0        & 0        & P_1^{-1}
  \end{bmatrix}>0, \label{eq.LMITestProofInvP1LMI}\\
  \begin{bmatrix}
      P_2^{-1}                      & *        & *        & * \\
  \sqrt{(1-p)c}(A+BK_2)P_2^{-1}     & P_2^{-1} & *        & * \\
         \sqrt{(1-p)(1-c)}AP_2^{-1} & 0        & P_2^{-1} & * \\
      \sqrt{p}AP_2^{-1}             & 0        & 0        & P_1^{-1}
  \end{bmatrix}>0. \label{eq.LMITestProofInvP2LMI}
\end{gather}
Left and right multiply \eqref{eq.LMITestProofInvP1LMI} with
$\mathrm{diag}(P_1, I, I, I)$, and left and right multiply
\eqref{eq.LMITestProofInvP2LMI} with $\mathrm{diag}(P_2, I, I, I)$, we obtain
\begin{gather*}
  \begin{bmatrix}
      P_{1}             & *        & *        & * \\
      \sqrt{qc}(A+BK_1) & P_2^{-1} & *        & * \\
       \sqrt{q(1-c)}A   & 0        & P_2^{-1} & * \\
      \sqrt{1-q} A      & 0        & 0        & P_1^{-1}
  \end{bmatrix}>0,\\
  \begin{bmatrix}
      P_2                   & *        & *        & * \\
      \sqrt{(1-p)c}(A+BK_2) & P_2^{-1} & *        & * \\
      \sqrt{(1-p)(1-c)}A    & 0        & P_2^{-1} & * \\
      \sqrt{p}A             & 0        & 0        & P_1^{-1}
  \end{bmatrix}>0.
\end{gather*}

In view of Schur complement lemma, we know that 

\begin{gather}
    \begin{multlined}[b][0.85\columnwidth]
   P_{1}- (1-q) A' P_{1}A - q(1-c)A'P_2A \\- qc (A+BK_1)'P_2(A+BK_1)>0,\label{eq.LyapunovLikeInequaltiy1}     
    \end{multlined} \\
    \begin{multlined}[b][0.85\columnwidth]
  P_{2}- pA' P_{1}A - (1-p)(1-c) A'P_2A \\- (1-p)c (A+BK_2)'P_2(A+BK_2)>0. \label{eq.LyapunovLikeInequaltiy2} \end{multlined}
\end{gather}

For any $P_2>0$ and $K$, we have
\begin{multline*}
  (A+BK)'P_2(A+BK) \\
    =A'P_2A- A'P_2B(B'P_2B)^{-1}B'P_2A \\
   + (K+(B'P_2B)^{-1}B'P_2A)'(B'P_2B)\\
    \times (K+(B'P_2B)^{-1}B'P_2A), 
\end{multline*}
which implies
\begin{multline*}
    A'P_2B(B'P_2B)^{-1}B'P_2A\\
    \ge  A'P_2A- (A+BK)'P_2(A+BK).
\end{multline*}
Therefore
\begin{multline*}
     -c A'P_2A+ c (A+BK)'P_2(A+BK)\\
    \ge  -cA'P_2B(B'P_2B)^{-1}B'P_2A,
\end{multline*}
for any $K$ and $P_2>0$.

In view of the above result and~\eqref{eq.LyapunovLikeInequaltiy1}
\eqref{eq.LyapunovLikeInequaltiy2}, we have,
\begin{gather}
    \begin{multlined}[b][0.85\columnwidth]
   \frac{P_{1}- (1-q) A' P_{1}A}{q}>A'P_2A\\-cA'P_2B(B'P_2B)^{-1}B'P_2A,\label{eq.Suffi1ForAllI} 
    \end{multlined}\\
    \begin{multlined}[b][0.85\columnwidth]
    \frac{P_{2}- pA' P_{1}A}{1-p} >  A'P_2A\\-cA'P_2B(B'P_2B)^{-1}B'P_2A. \label{eq.Suffi2ForAllI}
    \end{multlined}
\end{gather}
Since $-c=\min_k\max_i (\lambda_i^2 k^2 +2\lambda_i k)$ and the optimal $k$ to the minmax problem is $ \check{k}=-\frac{2}{\lambda_2+\lambda_N}$, we know that
\begin{gather*}
    \begin{multlined}[b][0.85\columnwidth]
   \frac{P_{1}- (1-q) A' P_{1}A}{q} > A'P_2A \\+ (\lambda_i^2\check{k}^2+2\lambda_i \check{k} )A'P_2B(B'P_2B)^{-1}B'P_2A,      
    \end{multlined}
  \\
  \begin{multlined}[b][0.85\columnwidth]
      \frac{P_{2}- pA' P_{1}A}{1-p} >  A'P_2A \\
      + (\lambda_i^2\check{k}^2+2\lambda_i \check{k})A'P_2B(B'P_2B)^{-1}B'P_2A 
  \end{multlined}
  \end{gather*}
hold for all $i=2,\ldots, N$. Therefore Theorem~\ref{thm.IffCondition}.1) is satisfied with
  \begin{gather*}
    P_{i,1} =P_1, P_{i,2}=P_2, K=\check{k}(B'P_2B)^{-1}B'P_2A.
  \end{gather*}
The proof is completed. 
\end{proof}

\subsection{Analytic Consensus Conditions}

The criterion stated in Theorem~\ref{thm.NumericalSufficieny} is easy to verify. However, it fails to provide insights into the consensusability problem. In the following, we provide analytical consensusability conditions, which show directly how the channel properties, the network topology and the agent dynamics interplay with each other to allow the existence of a distributed consensus controller. The following lemma is needed in proving the main result and is stated first.

\begin{lemma}[\cite{SchenatoL2007ProcIEEE}]~\label{lemma:MAREsolvability} Under Assumption~\ref{assumption:matrixAisUnstable}, if $(A, B)$ is controllable, then
    \begin{equation} \label{eq.MARE} P>A'PA-\gamma A'PB(B'PB)^{-1}B'PA
    \end{equation}
    admits a solution $P>0$, if and only if $\gamma$ is greater than a critical value $\gamma_c>0$.
\end{lemma}

\begin{remark} The value $\gamma_c$ is of great importance in determining the critical lossy probability in Kalman filtering over intermittent channels~\cite{SinopoliBruno2004TAC, SchenatoL2007ProcIEEE, MoYilin2008CDC}. It has been shown that the critical value $\gamma_c$ is only determined by the pair $(A, B)$ \cite{MoYilin2008CDC}. However, an explicit expression of $\gamma_c$ is only available for some specific situations. For example, when $\mathrm{rank}(B)=1$, $\gamma_c=1-\frac{1}{\Pi_i|\lambda_i(A)|^2}$ and when $B$ is square and invertible, $\gamma_c=1-\frac{1}{\max_i |\lambda_i(A)|^2}$. For other cases, the critical value $\gamma_c$ can be obtained by solving a quasiconvex LMI optimization problem~\cite{SchenatoL2007ProcIEEE}.
\end{remark}

\begin{theorem}\label{thm.ExplicitSufficiency}
 Under Assumption~\ref{assumption:matrixAisUnstable} and~\ref{assump.identicalPacketLoss}, the MAS~\eqref{eq:AgentDynamics} is mean square consensusable by the protocol~\eqref{eq:ConsensusProtocol} if
    \begin{align}
     \gamma_1=\min\{q, 1-p\}\left[ 1-\left( \frac{\lambda_N-\lambda_2}{\lambda_N+\lambda_2}
      \right)^2 \right]>\gamma_c,
      \label{eq.sufficientCondition}
    \end{align}
    where $\gamma_c$ is given in Lemma~\ref{lemma:MAREsolvability}. Moreover, if~\eqref{eq.sufficientCondition} holds, an admissible control gain is given by
    \begin{align*}
     K=-\frac{2}{\lambda_2+\lambda_N}(B'PB)^{-1}B'PA,
    \end{align*}
    where $P$ is the solution to~\eqref{eq.MARE} with $\gamma=\gamma_1$.  
\end{theorem}

\begin{proof}
If the~\eqref{eq.sufficientCondition} holds, in view of Lemma~\ref{lemma:MAREsolvability}, there exists a $P>0$ to~\eqref{eq.MARE} with $\gamma=\gamma_1$, such that
  \begin{align*}
  P&>A'PA-qcA'PB(B'PB)^{-1}B'PA,\\  
  P&>A'PA-(1-p)cA'PB(B'PB)^{-1}B'PA.
  \end{align*}
Since $-c= \max_i (\lambda_i^2\check{k}^2+2\lambda_i \check{k})$ with $\check{k}=-\frac{2}{\lambda_2+\lambda_N}$, we have
  \begin{align*}
  P&>A'PA+q(2\lambda_i \check{k}+\lambda_i^2 \check{k}^2)A'PB(B'PB)^{-1}B'PA,\\  
  P&>A'PA+(1-p)(2\lambda_i \check{k}+\lambda_i^2 \check{k}^2)A'PB(B'PB)^{-1}B'PA
    \end{align*}
    for all $i=2, \ldots, N$, which is the condition in Theorem~\ref{thm.IffCondition}.1) with
    \begin{gather*}
      P_{i,1}=P_{i,2}=P,\quad K=\check{k}(B'PB)^{-1}B'PA.
    \end{gather*}
    The proof is completed.
\end{proof}

\begin{remark}
    Theorem~\ref{thm.NumericalSufficieny} is obtained by letting $P_{i,1}=P_1$, $P_{i,2}=P_2$. Theorem~\ref{thm.ExplicitSufficiency} is obtained by letting $P_{i,1}=P_{i,2}=P$. Since the latter is more restrictive than the former. We can expect that Theorem~\ref{thm.ExplicitSufficiency}  is more restrictive than Theorem~\ref{thm.NumericalSufficieny}, which will be illustrated by a simulation example in the next subsection.
\end{remark}

In conjunction with the analytic sufficient consensusability condition in Theorem~\ref{thm.ExplicitSufficiency}, we also provide an explicit necessary consensusability condition as stated below. 

\begin{theorem} \label{thm.ExplicitNecessity}
 Under Assumption~\ref{assumption:matrixAisUnstable} and~\ref{assump.identicalPacketLoss},
the MAS~\eqref{eq:AgentDynamics} is mean square consensusable by the protocol~\eqref{eq:ConsensusProtocol} only if there exists $K$ such that
    \begin{gather}
      (1-q)^{\frac12}\rho(A)<1,\label{eq.explicitNecessityGeneral1}\\
      (1-p)^{\frac12}\rho(A+\lambda_i BK)<1 \label{eq.explicitNecessityGeneral2}
    \end{gather}
    for all $i=2, \ldots, N$. Moreover, when the agent is with single input, i.e., $m=1$, the MAS~\eqref{eq:AgentDynamics} is mean square consensusable by the protocol~\eqref{eq:ConsensusProtocol} only if
\begin{gather}
      (1-q)^{\frac12}\rho(A)<1, \label{eq.explicitNecessitySingleInput1}\\
(1-p)^{\frac{n}{2}} \det(A) \frac{\lambda_N-\lambda_2}{\lambda_N+\lambda_2}<1.\label{eq.explicitNecessitySingleInput2}
\end{gather}
\end{theorem}

\begin{proof}
If the MAS can achieve mean square consensus, in view of Theorem~\ref{thm.IffCondition}.1), we have that there exist $P_{i,1}>0$, $P_{i,2}>0$ and $K$ such that
\begin{gather*}
        P_{i,1}> (1-q) A' P_{i,1}A ,     \\
       P_{i,2} > (1-p) (A+ \lambda_i BK)' P_{i,2}(A+\lambda_iBK),
\end{gather*}
for all $i=2, \ldots, N$. Further from Lyapunov stability theory, we can obtain the necessary conditions~\eqref{eq.explicitNecessityGeneral1}, \eqref{eq.explicitNecessityGeneral2}.  

When the agent is with single input, following similar line  of argument as in the necessity proof of Lemma 3.1 in~\cite{YouKeyou2011TACconsensusability}, we can obtain the necessary condition~\eqref{eq.explicitNecessitySingleInput2}  from~\eqref{eq.explicitNecessityGeneral2}. The proof is completed. 
\end{proof}

\subsection{Critical Consensus Condition for Scalar Agent Dynamics}

When all the agents are with scalar dynamics, we can obtain a closed-form consensusability condition. The following lemma is needed in the proof of the main result and is stated first.

\begin{lemma}[\cite{XuLiang2017TCNS}]\label{lem:EquivalentConditios}
    Let $Q$ be defined in~\eqref{eq:twostateQ}; $D=\left[ \begin{smallmatrix}
            1& 0\\
            0 & \delta
        \end{smallmatrix}\right]$
    with $0<q, p, \delta<1$; $\lambda\in \mathbb{R}$, $|\lambda|\ge 1$. The
    following conditions are equivalent:
    \begin{enumerate}
      \item
        \begin{align*}
        \lambda^2 \rho (Q'D)<1,
        \end{align*}
      \item 
        \begin{gather}
            1-\lambda^2(1-q)>0, \label{eq:lemmacondition1} \\
            \lambda^2 \delta\left[ 1+\frac{p(\lambda^2-1)}{1-\lambda^2(1-q)}
            \right]<1. \label{eq:lemmacondition2}
        \end{gather}
    \end{enumerate}
     \end{lemma}

Without loss of generality, for scalar agent dynamics, i.e., $n=m=1$, we let $A=a\in \mathbb{R}$, $B=1$, $K=k \in \mathbb{R}$. The main result is stated as follows.
\begin{theorem}
\label{thm.ScalarIffConditoin}
Under Assumption~\ref{assumption:matrixAisUnstable} and~\ref{assump.identicalPacketLoss}, the MAS~\eqref{eq:AgentDynamics} with scalar agent dynamics is mean square consensusable by the protocol~\eqref{eq:ConsensusProtocol} if and only if
    \begin{gather}
        (1-q)a^2<1, \label{eq.ScalarIffCondtion1}\\
        a^2\left(\frac{\lambda_N-\lambda_2}{\lambda_N+\lambda_2}\right)^2 \left[1+\frac{p(a^2-1)}{1-a^2(1-q)}\right]<1\label{eq.ScalarIffCondition2}.
    \end{gather}
\end{theorem}

\begin{proof}
In view of Theorem~\ref{thm.IffCondition}.2), for scalar agent dynamics, the MAS~\eqref{eq:AgentDynamics} is mean square consensusable by the protocol~\eqref{eq:ConsensusProtocol} if and only if there exists $k$ such that 
\begin{align*}
 a^2\rho\left( Q'\times
  \begin{bmatrix}
      1 & 0\\
      0 & \frac{(a+\lambda_ik)^2}{a^2}
  \end{bmatrix}
\right)<1,
\end{align*}
for all $i=2, \ldots, N$. Further from Lemma~\ref{lem:EquivalentConditios}, a necessary and sufficient consensus condition is that if there exists $k$ such that for all $i=2, \ldots, N$.
     \begin{gather}
        (1-q)a^2<1,\label{eq.scalarIff1Temp} \\
       \left(a+\lambda_ik\right)^2 \left[1+\frac{p(a^2-1)}{1-a^2(1-q)}\right]<1. \label{eq.simultaneousRequirement}
    \end{gather}
    Since~\eqref{eq.simultaneousRequirement} holds for all $i$, we have that
    \begin{align*}
     \min_k\max_i(a+\lambda_ik)^2  \left[1+\frac{p(a^2-1)}{1-a^2(1-q)}\right]<1.
    \end{align*}
    Moreover, since 
    \begin{align*}
 \min_k\max_i(a+\lambda_ik)^2=a^2\left(\frac{\lambda_N-\lambda_2}{\lambda_N+\lambda_2}\right)^2,
    \end{align*}
    we can obtain the necessary and sufficient consensusability condition~\eqref{eq.ScalarIffCondtion1}, \eqref{eq.ScalarIffCondition2} from \eqref{eq.scalarIff1Temp}, \eqref{eq.simultaneousRequirement}. The proof is completed. 
\end{proof}

Interestingly, we can show that when the agent dynamics is scalar, the sufficient condition indicated in Theorem~\ref{thm.NumericalSufficieny} is also necessary.
Theorem \ref{thm.NumericalSufficieny} is equivalent to check the solvability of~\eqref{eq.Suffi1ForAllI} and \eqref{eq.Suffi2ForAllI}. For scalar systems with $A=a\in \mathbb{R}, B=1$,~\eqref{eq.Suffi1ForAllI} and~\eqref{eq.Suffi2ForAllI} change to
   \begin{gather}
     [1-(1-q)a^2]P_1>qa^2(1-c)P_2, \label{eq.ScalarExistenceCondition1}\\
     [1-(1-p)(1-c)a^2]P_2>pa^2P_1. \label{eq.ScalarExistenceConditoin2}
   \end{gather}
We can show that the necessary and sufficient condition to guarantee the solvability of the above
inequality is given by~\eqref{eq.ScalarIffCondtion1} and
\eqref{eq.ScalarIffCondition2}. Since $P_1>0$ and $qa^2(1-c)P_2>0$, we have from~\eqref{eq.ScalarExistenceCondition1} that $1-(1-q)a^2>0$, which gives~\eqref{eq.ScalarIffCondtion1}. Let $\theta=1-c=\left( \frac{\lambda_N-\lambda_2}{\lambda_N+\lambda_2}
\right)^2$. We can obtain a lower bound of $P_1$ from~\eqref{eq.ScalarExistenceCondition1} and
substitute this bound into~\eqref{eq.ScalarExistenceConditoin2} to obtain  
\begin{align*}
     [1-(1-p)\theta a^2]P_2>pa^2 \frac{qa^2\theta P_2}{[1-(1-q)a^2]}.
\end{align*}
Since $P_2>0$, we further have that
\begin{align*}
     [1-(1-p)\theta a^2][1-(1-q)a^2]-pqa^4\theta>0,
\end{align*}
which implies
\begin{align*}
a^2 \theta[(1-p)-a^2 (1-p-q)]<1-a^2(1-q).
\end{align*}
Dividing both sides by $1-a^2(1-q)$, we can obtain~\eqref{eq.ScalarIffCondition2}.

In contrast to the tightness of Theorem~\ref{thm.NumericalSufficieny} for scalar systems, Theorem~\ref{thm.ExplicitSufficiency} is generally not necessary. Consider the case that $A=2$, $B=1$, $\lambda_2=2$ and $\lambda_N=3$, then the tolerable $(p, q)$ from Theorem~\ref{thm.ScalarIffConditoin} are given by
   \begin{gather*}
       q>\frac34, \quad p<7\times (q-\frac34).
   \end{gather*}
   While the sufficiency indicated by Theorem~\ref{thm.ExplicitSufficiency}, is given by
   \begin{gather*}
       q> \frac{25}{32}, \quad p< \frac{7}{32}. 
   \end{gather*}
   The tolerable failure rate and recovery rate are plotted in Fig.~\ref{fig.scalarCompare}. It is clear that the result in Theorem~\ref{thm.ExplicitSufficiency} is conservative in the case of scalar agent dynamics. 
\begin{figure}[htpb] \centering
    \includegraphics[width=0.4\textwidth]{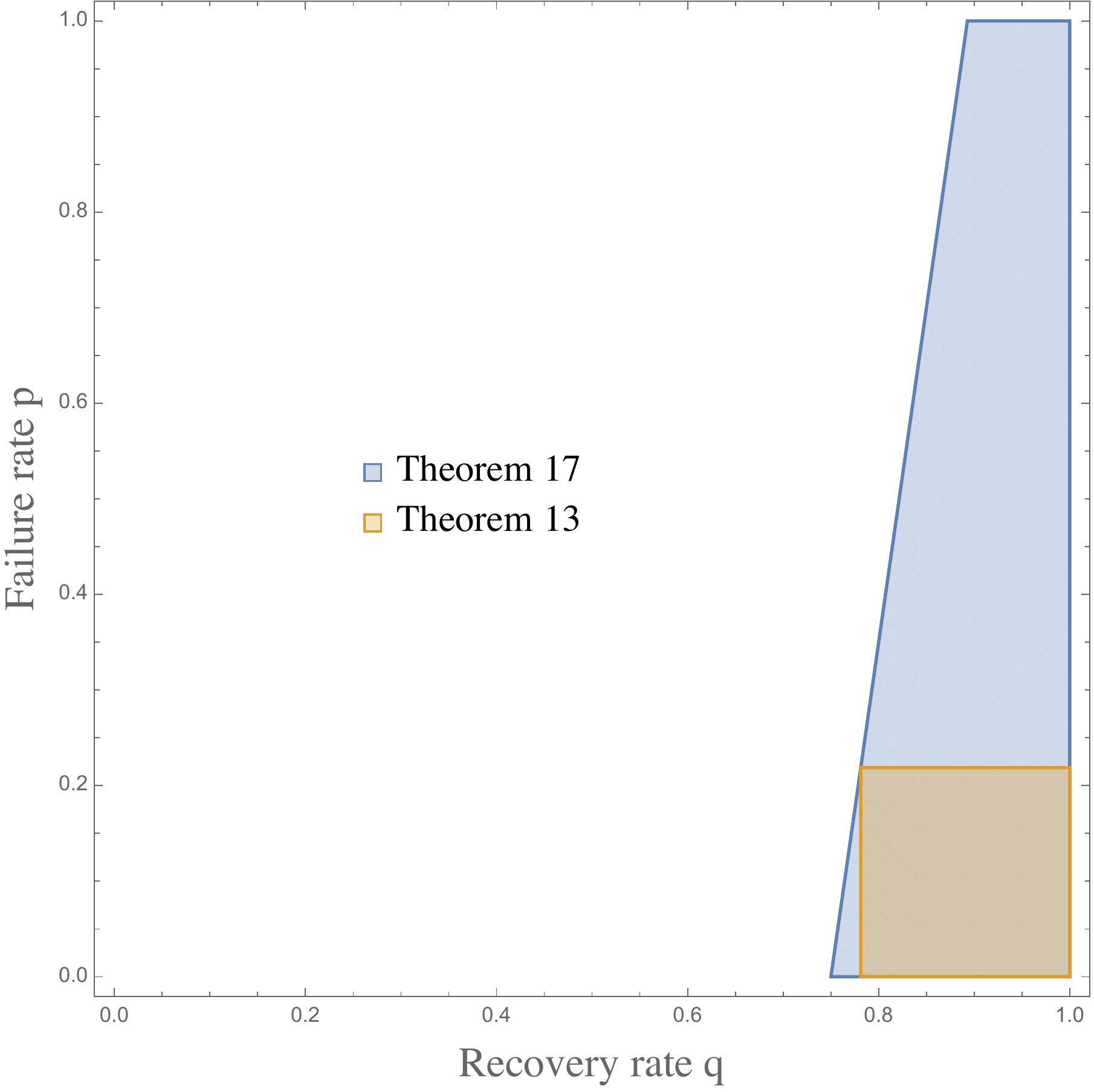}\\
    \caption{Tolerable failure rate and recovery rate}
    \label{fig.scalarCompare}
\end{figure}

The assumption of identical channel loss distributions is somewhat restrictive and less practical. However, it is the simplest case in studying the consensus problem over Markovian packet loss channels and is expected to shed light on solutions to more general nonidentical cases, which is studied in the subsequent section.

\section{Nonidentical Markovian Packet Loss\label{sec.NonidenticalMarkovLoss}}

In the presence of nonidentical packet losses, the consensus error dynamics of $\delta$ is given by $\delta (t+1)=\left( I \otimes A+ \mathcal{L}(t)\otimes BK \right)\delta (t) $ with $\mathcal{L}(t)$ modeling both the communication topology and the packet losses. Since the packet loss is coupled with the communication topology in $\mathcal{L}(t)$, the analysis of the mean square consensus is difficult. Therefore, the edge Laplaican~\cite{ZelazoD2011TAC} is used to model the consensus error dynamics as in~\cite{XuLiang2016Automatica}, which allows to separate the lossy effect from the network topology to facilitate the consensusability analysis by building dynamics on edges rather than on vertexes.

The following graph definitions are needed in introducing the edge Laplacian. A virtual orientation of the edge in an undirected graph is an assignment of directions to the edge $(i,j)$ such that one vertex is chosen to be the initial node and the other to be the terminal node. The incidence matrix $E$ for an oriented graph $\mathcal{G}$ is a $\{0, 1, -1\}$-matrix with rows and columns indexed by vertices and edges of $\mathcal{G}$, respectively, such that
\begin{equation*}
{[E]}_{ik}=\left\{
\begin{aligned}
+1, & \;\;\text{if $i$ is the initial node of edge $k$}\\
-1, & \;\;\text{if $i$ is the terminal node of edge $k$}\\
0, & \;\;\text{otherwise }
\end{aligned}
\right.
\end{equation*}
The graph Laplacian $\mathcal{L}$ and edge Laplacian $\mathcal{L}_e$ can be constructed from the incidence matrix respectively as $ \mathcal{L}=EE' $, $\mathcal{L}_e=E'E $~\cite{ZelazoD2011TAC}.

Define the state on the $i$-th edge as $ z_i=x_j-x_k $, with $j,k$ representing the initial node and the terminal node of the $i$-th edge, respectively. Assume that the packet losses on the same edge are equal, i.e., $\gamma_{jk}=\gamma_{kj}$, which make sense in practical applications~\cite{DeyS2009Automatica}. Following the definition of incidence matrix, the controller~\eqref{eq:ConsensusProtocol} can be alternatively represented as
\begin{align*}
 u_j(t)=K\sum_{k=1}^{l}e_{jk}\zeta_k(t)z_k(t),
\end{align*}
where $l$ is the total number of edges in $\mathcal{G}$, $e_{jk}$ is the $jk$-th element of $E$ and $\zeta_k$ denotes the packet loss effect on the $k$-th edge, i.e., $\zeta_k=\gamma_{ij}$ where $i, j$ are the initial node and terminal node of the $k$-the edge. If we define $z=[z_1',z_2',\ldots,z_{l}']'$, then following similar steps as in~\cite{XuLiang2016Automatica}, the closed-loop dynamics on edges can be calculated as
\begin{equation}
\label{dynamicsOfZ}
z(t+1)=\left( I\otimes A +\mathcal{L}_e \zeta(t)\otimes BK \right)z(t)
\end{equation}
with $\zeta(t)=\mathrm{diag}(\zeta_1(t),\zeta_2(t),\ldots,\zeta_{l}(t))$.

With appropriate indexing of edges, we can write the incidence matrix $E$ as $E=[E_{\tau}, E_c] $, where edges in $E_\tau$ are on a spanning tree and edges in $E_c$ complete cycles in $\mathcal{G}$. We further have that when $\mathcal{G}$ is connected, there exists a matrix $T$, such that $E_c=E_\tau T$~\cite{ZelazoD2011TAC}. Moreover, with such indexing of edges, we can decompose the edge state $z$ as $z=[z_\tau', z_c']'$,  where $z_\tau$ is  the edge state on the spanning tree and $z_c$ is the remaining edge state. Besides, it is straightforward to verify that $z_c= (T'\otimes I)
z_\tau$, since $z=[z_\tau', z_c']'=(E'\otimes I)x=([E_\tau, E_c]'\otimes I)x$ and $E_c=E_\tau T$. Let $M=E_\tau'E$ and $R=[I, T]$, we have that
\begin{align}
   &z_{\tau}(t+1)=(I\otimes A) z_{\tau}(t) +((E_{\tau}'E_{\tau}\zeta_\tau(t))\otimes (BK))
    z_{\tau}(t) \nonumber \\
  &  \quad + ((E_{\tau}'E_{c}\zeta_{c}(t))\otimes (BK)) z_{c}(t) \nonumber \\
  &=(I\otimes A+ (E_{\tau}'E_{\tau}\zeta_{\tau}(t)+E_{\tau}'E_c\zeta_c(t)T')\otimes
    (BK))z_{\tau}(t) \nonumber \\
    & = (I\otimes A + (M \zeta(t) R')\otimes (BK))z_{\tau}(t),\label{eq.reducedTreeEdgeDynamics}
\end{align}
where $\zeta_\tau, \zeta_c$ represent the packet losses on tree edges and cycle edges, respectively. The MAS can achieve mean square consensus if and only if~\eqref{eq.reducedTreeEdgeDynamics} is mean square stable.

The possible sample space of $\zeta(t)$ is $\Phi=\{\Lambda_0, \ldots, \Lambda_{2^l-1}\}$, where the $i$-th element $\Lambda_i$ is $\Lambda_i=\mathrm{diag}(\eta_1, \ldots, \eta_l)$ with $\eta_j \in \{0,1\}$, $j=1, \ldots, l$, being the $j$-th component of the binary expansion of $i$, i.e., $i=\eta_l2^{l-1}+\ldots+\eta_12^0$. We make the following assumptions for the packet loss matrix $\zeta(t)$.

\begin{assumption}
    \label{assump:nonIdenticalLoss}
The packet loss process $\{\zeta(t)\}_{t\ge 0}$ is a time-homogeneous Markov stochastic process, which has $o$ states $\{\Gamma_1, \ldots, \Gamma_o\}$, where $\Gamma_i\in \Phi$. The probability transition matrix $Q$ is an $o\times o$ matrix with the $ij$-th element being $p_{ij}$. 
\end{assumption}

\begin{remark}
It is possible that certain outcomes in $\Phi$ are unlikely to happen. For example, if two agents are close to each other, the communication between them can be reliable. It is unlikely that the communication link would undergo packet losses. In such cases, the sample space of $\zeta(t)$ would be a subset of $\Phi$. Therefore, in Assumption~\ref{assump:nonIdenticalLoss}, we use $o$ to denote the carnality of the actual sample space of $\zeta(t)$, which might be smaller than $2^{l}$.         
\end{remark}

Therefore~\eqref{eq.reducedTreeEdgeDynamics} is a Markov jump linear system. In view of Theorem 3.9 in~\cite{CostaO2005Book}, we have the following consensus result.

\begin{theorem}\label{thm:NonidentcalIff}
Under Assumption~\ref{assumption:matrixAisUnstable} and~\ref{assump:nonIdenticalLoss}, the MAS~\eqref{eq:AgentDynamics} is mean square consensusable by the protocol~\eqref{eq:ConsensusProtocol} if and only if either of the following conditions holds, where $\mathcal{S}_i(K)= (I\otimes A+M\Gamma_iR' \otimes BK)$

\begin{enumerate}
  \item there exist $P_i>0$, $i=1, \ldots, o$ and $K$ such that
\begin{align*}
 P_i>\sum_{j=1}^{o}p_{ij}\mathcal{S}_j(K)' P_j\mathcal{S}_j(K)
\end{align*}
for all $i=1, \ldots, o$. 
\item there exists $K$ such that
  \begin{align*}
    \rho\left( (Q'\otimes I)\mathrm{diag}\left(\mathcal{S}_i(K)\otimes \mathcal{S}_i(K) \right)  \right) <1.
  \end{align*}
\end{enumerate}
\end{theorem}

We can show that the consensus criterion in Theorem~\ref{thm:NonidentcalIff}.1) is equivalent to a feasibility  problem with BMI constraints. 
Therefore, checking the conditions in Theorem~\ref{thm:NonidentcalIff} are generally not easy. In the following a numerically easy testable condition in terms of the feasibility of LMIs are proposed.  

\begin{theorem}
Under Assumption~\ref{assumption:matrixAisUnstable} and~\ref{assump:nonIdenticalLoss}, the MAS~\eqref{eq:AgentDynamics} is mean square consensusable by the protocol~\eqref{eq:ConsensusProtocol} if there exists $\kappa\in \mathbb{R}$ such that the following LMIs are feasible, 
    \begin{align}
      \label{eq:NonidenticalLMISuff}
  \begin{bmatrix}
      -I & \kappa V_i'\\
      \kappa V_i & \kappa N_i+\gamma_c I
  \end{bmatrix}<0
\end{align}
for all $i=1, \ldots, o$, where $\gamma_c$ is given in Lemma~\ref{lemma:MAREsolvability},   $N_i=\sum_{j=1}^{o}p_{ij}( R\Gamma_jM' + M\Gamma_jR' )$, $M_i=  \sum_{j=1}^{o} p_{ij}  R\Gamma_jM'M\Gamma_jR'$ and $V_i$ is the Cholesky decomposition of $M_i$, i.e., $M_i=V_iV_i'$. Moreover, if~\eqref{eq:NonidenticalLMISuff} is satisfied, a control gain is given by $K=\kappa (B'PB)^{-1}B'PA$ where $P$ is the solution of~\eqref{eq.MARE} with $\gamma=\min_i \lambda_{\min}(-\kappa N_i-\kappa^2M_i)$.    
\end{theorem}

\begin{proof}
If~\eqref{eq:NonidenticalLMISuff} holds, there exists $\kappa$ such that $\kappa N_i +\kappa^2M_i<-\gamma_c I$ for all $i=1, \ldots, o$. Since $\kappa N_i +\kappa^2M_i$ is real and symmetric, it is diagonalizable by an orthogonal matrix $\Psi$, i.e., $\Psi'(\kappa N_i +\kappa^2M_i)\Psi= \Upsilon$ and $\Upsilon$ is diagonal. Then we have that $\Upsilon< -\gamma_cI $. In view of Lemma~\ref{lemma:MAREsolvability}, we can find $P>0$ such that
\begin{align*}
 I\otimes P>I\otimes A'PA+ \Upsilon \otimes A'PB(B'PB)^{-1}B'PA. 
\end{align*}
Left and right multiply the above inequality with $\Psi\otimes I$ and $\Psi' \otimes I$, we have that
\begin{align*}
 I\otimes P>I\otimes A'PA+ (\kappa N_i +\kappa^2M_i) \otimes A'PB(B'PB)^{-1}B'PA. 
\end{align*}
From the definitions of $N_i$ and $M_i$ and the relation that $\sum_{j=1}^op_{ij}=1$, we further have that
\begin{multline*}
 I\otimes P>\sum_{j=1}^{o}p_{ij}(I\otimes A'PA+(\kappa R\Gamma_jM' + \kappa M\Gamma_jR' \\+\kappa^2 R\Gamma_jM'M\Gamma_jR')\otimes A'PB(B'PB)^{-1}B'PA), 
\end{multline*}
which is the sufficient condition given in Theorem~\ref{thm:NonidentcalIff}.1) with  $P_1=\ldots=P_{o}=I\otimes P$ and $K=\kappa (B'PB)^{-1}B'PA$. The proof is completed. \end{proof}



\begin{remark}
   This paper only discusses the consensusability  problem over undirected graphs. For the consensusability problem with directed graphs, the compressed edge Laplacian proposed in~\cite{XuLiang2018AutomaticaAccepted} can be used to model the consensus error dynamics. Then following similar derivations as in this section, consensus conditions over directed graphs in the presence Markovian packet losses can be obtained.   
\end{remark}

\section{Numerical Simulations \label{sec.Simulations}}

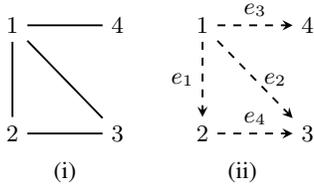
\begin{figure} \centering
    \begin{tabular}{cc}
      \begin{tikzpicture} [line width=0.7pt,>=stealth]
          \node (node1) {$1$};
          \node (node2) [below =of node1] {$2$};
          \node (node3) [right =of node2] {$3$};
          \node (node4) [right =of node1] {$4$};
          \path (node1) edge  (node2);
          \path (node1) edge  (node3);
          \path (node1) edge  (node4);
          \path (node3) edge  (node2);
      \end{tikzpicture}&
                         \begin{tikzpicture} [line width=0.7pt,>=stealth]
                             \node (node1) {$1$};
                             \node (node2) [below =of node1] {$2$};
                             \node (node3) [right =of node2] {$3$};
                             \node (node4) [right =of node1] {$4$};
                             \path[dashed, ->] (node1) edge node [left] {$e_1$} (node2);
                             \path[dashed, ->] (node1) edge node [right] {$e_2$} (node3);
                             \path[dashed, ->] (node1) edge node [above] {$e_3$} (node4);
                             \path[dashed, ->] (node2) edge node [above] {$e_4$} (node3);
                         \end{tikzpicture} 
      \\ (i) & (ii) 
    \end{tabular}
    \caption{Communication graphs used in simulations: (i) an undirected graph (ii)
      applying an orientation to edges in (i)}\label{fig.directedGraphsInSimulation}
\end{figure}

In this section, simulations are conducted to verify the derived results. In simulations, agents are assumed to have system parameters
\begin{gather*}
 A= \begin{bmatrix}
      1.1830 &  -0.1421 &  -0.0399\\
    0.1764  &  0.8641 &  -0.0394\\
    0.1419  & -0.1098 &   0.9689
\end{bmatrix},
B= \begin{bmatrix}
 0.1697  &  0.3572\\
    0.5929 &   0.5165\\
    0.1355 &   0.9659
\end{bmatrix}.
\end{gather*}
The initial state of each agent is uniformly and randomly generated from the interval $(0, 0.5)$. We assume that there are four agents and the undirected communication topology among agents is given in Fig.~\ref{fig.directedGraphsInSimulation}.(i). We first consider the consensus with identical Markovian packet losses. The Markov packet losses in transmission channels are assumed to have parameters $p=0.2$, $q=0.7$. With such configurations, the LMIs in Theorem~\ref{thm.NumericalSufficieny} are feasible and an admissible  control parameter is given by 
\begin{align*}
 K= \begin{bmatrix*}
 2.0646 &  -1.3157 &   -0.0939 \\
   -0.5767 &    0.2947   &-0.3324
  \end{bmatrix*}.
\end{align*}
 
 The simulation results are presented by averaging over 1000 runs. Mean square consensus errors for agent 1 are plotted in Fig.~\ref{fig.simulation3ResultAgent1ConsensusError}, which shows that the mean square consensus is achieved.

\begin{figure}[!htpb] \centering
    \includegraphics[width=0.4\textwidth]{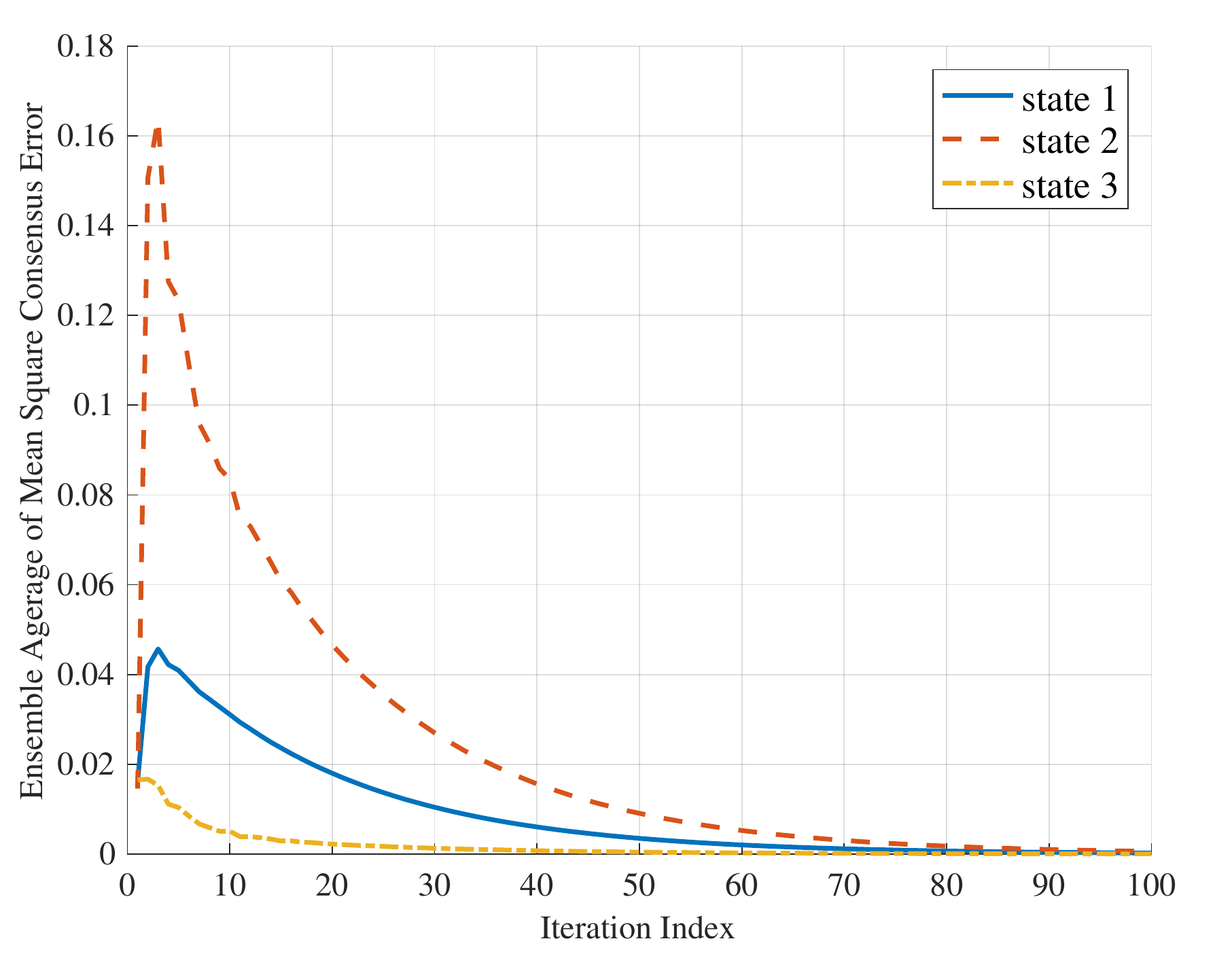}
    \caption{Mean square consensus error for agent 1 under identical packet losses}\label{fig.simulation3ResultAgent1ConsensusError}
\end{figure}

Secondly, we consider the consensus over nonidentical Markovian packet loss networks. We index the edges and apply a virtual orientation to each edge as in Fig.~\ref{fig.directedGraphsInSimulation}.(ii). Denote the packet loss processes in these edges as $\zeta_1(t), \zeta_2(t), \zeta_3(t), \zeta_4(t)$. Suppose the time-homogeneous Markov packet loss process $\{\zeta(t)\}_{t\ge0}$ with $\zeta(t)=\mathrm{diag}(\zeta_1(t), \zeta_2(t), \zeta_3(t), \zeta_4(t))$ has three states
 $\Gamma_1=\mathrm{diag}(1,0,1,0)$, $\Gamma_2=\mathrm{diag}(0,1,0,1)$, $\Gamma_3=\mathrm{diag}(1,1,1,1)$ and is with the probability transition matrix
\begin{align*}
 Q= \begin{bmatrix}
 0.3811 &   0.1446 &    0.4743\\
    0.2445 &    0.5121 &    0.2434\\
    0.5390 &    0.0215 &    0.4395
  \end{bmatrix}.
\end{align*}
With such settings, we can show that~\eqref{eq:NonidenticalLMISuff} is feasible, and an admissible control gain is given by
\begin{align*}
  K=
  \begin{bmatrix}
    1.7394  & -1.3873 &   0.0771\\
   -0.2133    &0.2212   &-0.5269
  \end{bmatrix}.
\end{align*}
The simulation results are presented by averaging over 1000 runs. The consensus error for agent $1$ is given below, which shows that the mean square consensus is achieved.

\begin{figure}[!htpb] \centering
    \includegraphics[width=0.4\textwidth]{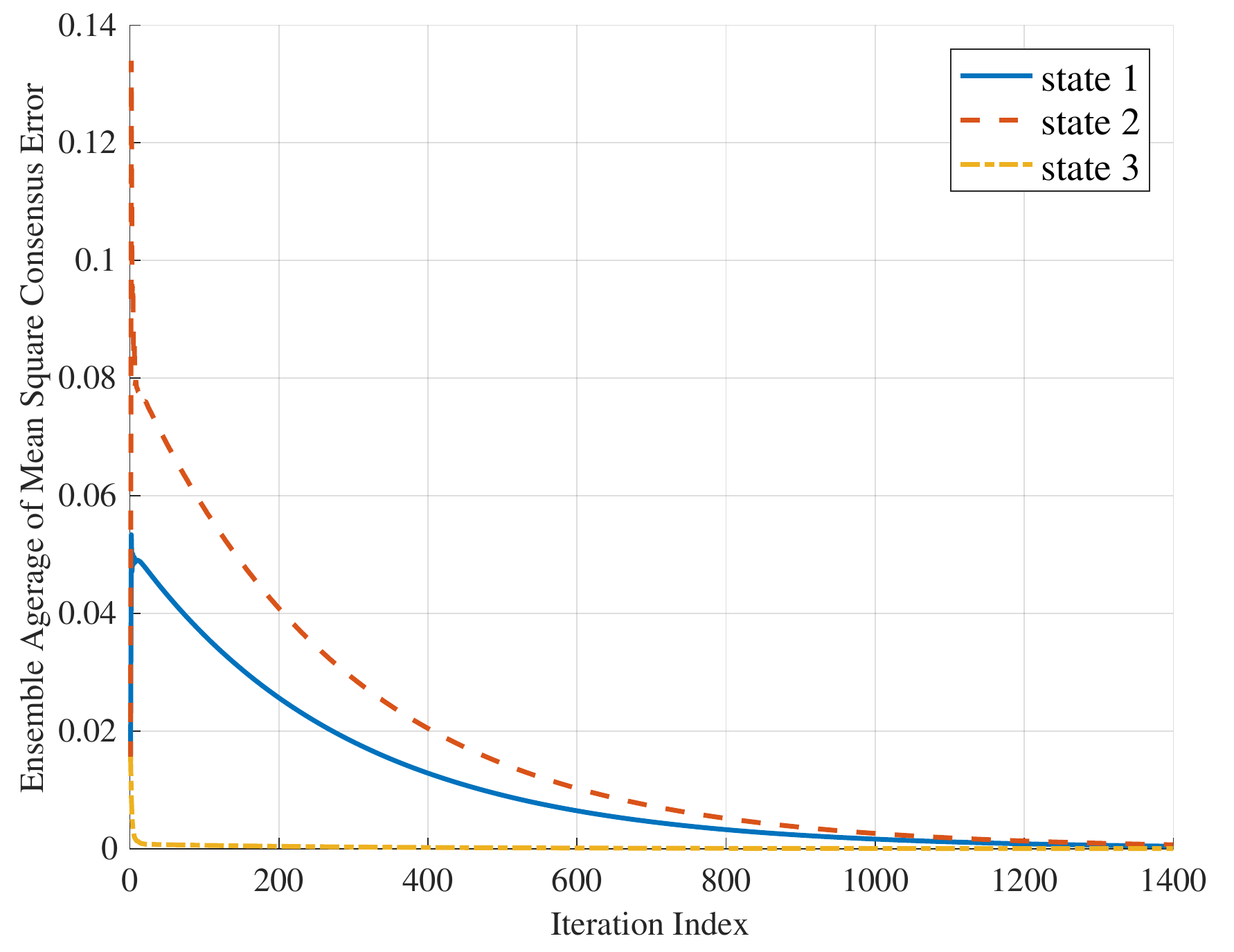}
    \caption{Mean square consensus error for agent 1 under nonidentical packet losses}\label{fig.simulationNonIdenticalResultAgent1ConsensusError}
\end{figure}

\section{Conclusions \label{sec.Conclusion}}

This paper studies the mean square consensusability problem of MASs over Markovian packet loss channels. Necessary and sufficient consensus conditions are derived under various situations. The derived results show how the agent dynamics, the network topology and the channel loss interplay with each other to allow the existence of a linear distributed consensus controller. Analytic consensus conditions are only provided for consensus with identical Markovian packet losses. The case with nonidentical Markovian packet losses deserves more effort.

\bibliographystyle{ieeetr}

 \bibliography{/Users/xuliang/Dropbox/xuliang-research/4-references/bibtex-refs_bt/references}

\end{document}